\documentclass[12pt]{article}      
\usepackage{amsmath,amssymb,amsthm, amscd, graphicx, verbatim, setspace} 
\usepackage{setspace}

\theoremstyle{plain}
\newtheorem{thm}{Theorem}
\newtheorem{lem}[thm]{Lemma}
\newtheorem{cor}[thm]{Corollary}

\title{Bounds for approximating lower envelopes with polynomials of degree at most $d$}
\date{}
\author{Jesse Geneson\\
\small\tt geneson@gmail.com\\
}
\begin{document}
\maketitle

\begin{abstract}
Given a lower envelope in the form of an arbitrary sequence $u$, let $LSP(u, d)$ denote the maximum length of any subsequence of $u$ that can be realized as the lower envelope of a set of polynomials of degree at most $d$. Let $sp(m, d)$ denote the minimum value of $LSP(u, d)$ over all sequences $u$ of length $m$. We derive bounds on $sp(m, d)$ using another extremal function for sequences.

A sequence $u$ is called $v$-free if no subsequence of $u$ is isomorphic to $v$. Given sequences $u$ and v, let $LSS(u, v)$ denote the maximum length of a $v$-free subsequence of $u$. Let $ss(m, v)$ denote the minimum of $LSS(u, v)$ over all sequences $u$ of length $m$. By bounding $ss(m, v)$ for alternating sequences $v$, we prove quasilinear bounds in $m^{1/2}$ on $sp(m,d)$ for all $d > 0$.


\end{abstract}

\section{Introduction}

A graph $G$ contains a graph $H$ if some subgraph of $G$ is isomorphic to $H$. Otherwise $G$ avoids $H$. If $G$ and $H$ are graphs, $LSG(G, H)$ denotes the number of edges in the largest subgraph of $G$ that avoids $H$. Let $sg(m, H)$ denote the minimum of $LSG(G, H)$ over all graphs $G$ with $m$ edges. The function $sg(m, H)$ has been bounded for complete bipartite graphs in \cite{fox} and for large subgraphs without short cycles in \cite{fouc}.

In this paper, we investigate a function about lower envelopes that is analogous to $sg$: Given a lower envelope represented as an arbitrary sequence $u$, how much of the lower envelope can be realized using only polynomials of degree at most $k$? 

This problem is closely related to Davenport-Schinzel sequences and finding the maximum possible complexity of a lower envelope of polynomials of degree at most $k$ \cite{ds}. Researchers have proved tight bounds on the maximum possible length of Davenport-Schinzel sequences \cite{shor, niv, pettie}, but the bounds for the maximum possible complexity of a lower envelope of $n$ polynomials of degree at most $k$ are still not tight for $k \geq 3$. 

Shor proved that the maximum possible complexity of a lower envelope of $n$ polynomials of degree at most $4$ is $\Omega(n \alpha(n))$ \cite{shorm}, where $\alpha(n)$ denotes the incredibly slow-growing inverse Ackermann function. However, it is conjectured that the maximum possible complexity is $\Theta(n 2^{\alpha(n)})$ for $k = 4$ and $\Theta(n \alpha(n))$ for $k = 3$. For $k = 1, 2$, it is easy to show that the bounds are $\Theta(n)$. 

Let $LSP(u, k)$ denote the maximum length of any subsequence of $u$ that can be realized as the lower envelope of a set of polynomials of degree at most $k$. Let $sp(m, k)$ denote the minimum value of $LSP(u, k)$ over all sequences $u$ of length $m$. In order to get bounds for this problem, we find bounds on corresponding functions for forbidden $0-1$ matrices and sequences. 

A $0-1$ matrix $A$ contains a $0-1$ matrix $P$ if some submatrix of $A$ is either equal to $P$ or can be changed to $P$ by turning some ones into zeroes. If $P$ is a $0-1$ matrix and $A$ is a $0-1$ matrix, let $LSM(A, P)$ denote the maximum number of ones in a $P$-avoiding $0-1$ matrix $B$ that is contained in $A$. Let $sm(m, P)$ denote the minimum of $LSM(A, P)$ over all $0-1$ matrices $A$ with $m$ ones. 

We identify all matrices $P$ for which $ex(n, P) = O(n)$ and $sm(m, P) \neq \Theta(m^{1/2})$, where $ex(n, P)$ is the maximum number of ones in an $n \times n$ $0-1$ matrix that avoids $P$. For every $0-1$ matrix $P$ such that $ex(n, P) = O(n)$, it is easy to show that $sm(m, P) = O(m^{1/2})$.

A sequence $u$ is called $v$-free if no subsequence of $u$ is isomorphic to $v$. Given sequences $u$ and v, let $LSS(u, v)$ denote the maximum length of a $v$-free subsequence of $u$. Let $ss(m, v)$ denote the minimum of $LSS(u, v)$ over all sequences $u$ of length $m$. 

We say that a sequence $u$ is realized as a lower envelope of a set of functions $\left\{f_{1}, \ldots, f_{n} \right\}$ if $\min \left\{f_{1}(x), \ldots, f_{n}(x) \right\}$ can be expressed as a piecewise function $g(x)$ such that the sequence of functions in $g(x)$ from $-\infty$ to $\infty$ is isomorphic to $u$.

Usually a function $f(n)$ is called \emph{quasilinear} if there exists a constant $t$ such that $f(n) = O( n \log(n)^{t})$. However in this paper, all of the quasilinear functions will satisfy an even stronger property: for each such function $f(n)$, there will exist a constant $t$ such that $f(n) = O(n 2^{\alpha(n)^{t}})$.

It is easy to see that $sp(m, k) \leq ss(m, a_{k+2})$, where $a_{n}$ denotes the alternation $a b a b \ldots$ of length $n$, since two polynomials of degree $k$ can intersect at most $k$ times. We use this inequality to show that $sp(m, k)$ is quasilinear in $m^{1/2}$ for every $k > 0$. 

\section{$0-1$ Matrix Results}

The first result is analogous to a result about complete bipartite graphs \cite{fox}. It is included to introduce a useful probabilistic technique that applies to other $0-1$ matrices, and to use for corollaries about extremal functions of other $0-1$ matrices.

\begin{thm}
If $P$ is an $r \times s$ $0-1$ matrix with all entries equal to $1$ and $s \geq r \geq 2$, then $sm(m, P) = \Theta(m^{r/(r+1)})$.
\end{thm}

This result follows from the next two lemmas, very similar to the proofs in \cite{fox}.

\begin{lem}\label{prob}
If $P$ is an $r \times r$ $0-1$ matrix with all entries equal to $1$ and $r \geq 2$, then $sm(m, P) = \Omega(m^{r/(r+1)})$.
\end{lem}

\begin{proof}
Let $A$ be any $0-1$ matrix with $m$ ones. First we claim that the maximum possible number of copies of $P$ in $A$ is less than $m^{r}$. To see this, note that any copy of $P$ is defined by the $r$ ones on its main diagonal. Among the $m$ ones in $A$, there are at most ${m \choose r}$ ways to choose $r$ ones, which is less than $m^{r}$.

Now create a matrix $A'$ from $A$ by not changing the zeroes in $A$, and changing each one in $A$ to zero with probability $1-p$, such that $p = \frac{1}{2}m^{-1/(r+1)}$. The expected number of ones in $A'$ is $m p$ and the expected number of copies of $P$ in $A'$ is at most $p^{r^{2}} m^{r}$. Then there is a $P$-free $0-1$ matrix contained in $A$ that has at least $m p - p^{r^{2}} m^{r} = \Omega(m^{r/(r+1)})$ ones on average. Therefore there exists a choice of $A'$ which produces a $P$-free subgraph of size $\Omega(m^{r/(r+1)})$.
\end{proof}

\begin{lem}
If $P$ is an $s \times r$ $0-1$ matrix with all entries equal to $1$ and $s \geq r \geq 2$, then $sm(m, P) = O(m^{r/(r+1)})$.
\end{lem}

\begin{proof}
Let $A$ be an $m^{r/(r+1)} \times m^{1/(r+1)}$ matrix with every entry equal to $1$, and let $A'$ be a $P$-free $0-1$ matrix contained in $A$.

Let $t$ be the number of ones in $A'$, let $R$ denote the rows of $A'$, and let $w_{i}$ denote the number of ones in row $i$ of $A'$ for each $i \in R$. By convexity, $m^{r/(r+1)} \times {t/m^{r/(r+1)} \choose r} \leq \sum_{i \in R} {w_{i} \choose r}$. By the pigeonhole principle, $\sum_{i \in R} {w_{i} \choose r} < s {m^{1/(r+1)} \choose r}$. 
\end{proof}

The next two lemmas give operations on $0-1$ matrices that only change $sm(m, P)$ by at most a constant factor, with some corollaries.

\begin{lem}
Suppose that $P$ is a $0-1$ matrix with a row $r$ that has a pair of adjacent ones in columns $c$ and $c+1$. Let $R$ be the $0-1$ matrix obtained by inserting a new column between columns $c$ and $c+1$ with a one in row $r$ and zeroes elsewhere. Then $sm(m,P) \leq sm(m, R) \leq 2sm(m, P)$.
\end{lem}

\begin{proof}
Let $A$ be a $0-1$ matrix with $m$ ones such that every $0-1$ matrix contained in $A$ with at least $sm(m,P)+1$ ones contains $P$. Suppose that $A'$ is a $0-1$ matrix contained in $A$ with $2sm(m, P)+1$ ones. Let $B$ be obtained from $A'$ by changing every other one in each row to zeroes, starting with the second one in each row. Since $B$ has at least $sm(m, P)+1$ ones, choose a copy of $P$ in $B$. This corresponds to a copy of $R$ in $A'$.
\end{proof}

\begin{cor}
If $P$ is an $r \times s$ $0-1$ matrix with all entries equal to $1$ and $s \geq r \geq 2$ and $R$ is obtained by inserting a new column between two columns of $P$ with a single one in the new column, then $sm(m, R) = \Theta(m^{r/(r+1)})$.
\end{cor}

\begin{lem}
Suppose that $P$ is a $0-1$ matrix with ones in both its bottom left and top right corners. Let $R$ be the $0-1$ matrix obtained by joining two copies of $P$ so that they overlap at the bottom left and top right corners and filling with zeroes elsewhere. Then $sm(m,P) \leq sm(m, R) \leq 2sm(m, P)$.
\end{lem}

\begin{proof}
Let $A$ be a $0-1$ matrix with $m$ ones such that every $0-1$ matrix contained in $A$ with at least $sm(m,P)+1$ ones contains $P$. Suppose that $A'$ is a $0-1$ matrix contained in $A$ with $2sm(m, P)+1$ ones. Choose a copy of $P$ in $A'$ and delete the one in the bottom right corner of $P$. Do this a total of $sm(m, P)+1$ times. The deleted entries form a copy of $P$, so $A'$ contains $R$.
\end{proof}

\begin{cor}
If $P$ is an $r \times s$ $0-1$ matrix with all entries equal to $1$ and $s \geq r \geq 2$ and $R$ is obtained by joining multiple copies of $P$ at their bottom left corners and top right corners and filling with zeroes elsewhere, then $sm(m, R) = \Theta(m^{r/(r+1)})$.
\end{cor}

Both of the operations in the last two lemmas are analogous to operations for the Tur\'{a}n extremal function $ex(n, P)$ \cite{kes}. The next lemma shows that the probabilistic technique from Lemma \ref{prob} can give tight lower bounds on $sm(m, P)$ for other $0-1$ matrices $P$.

\begin{lem}\label{lshape}
If $P = \begin{bmatrix} 1 & 0 \\ 1 & 1 \end{bmatrix}$, then $sm(m, P) = \Omega(m^{1/2})$.
\end{lem}

\begin{proof}
Let $A$ be any $0-1$ matrix with $m$ ones. The maximum possible number of copies of $P$ in $A$ is less than $m^{2}$. To see this, note that any copy of $P$ is defined by the $2$ ones on its main diagonal. Among the $m$ ones in $A$, there are at most ${m \choose 2}$ ways to choose $2$ ones, which is less than $m^{2}$.

Now create a matrix $A'$ from $A$ by not changing the zeroes in $A$, and changing each one in $A$ to zero with probability $1-p$, such that $p = \frac{1}{2}m^{-1/2}$. The expected number of ones in $A'$ is $m p$ and the expected number of copies of $P$ in $A'$ is at most $p^{3} m^{2}$. Then there is a $P$-free $0-1$ matrix contained in $A$ which has at least $m p - p^{3} m^{2} = \Omega(m^{1/2})$ edges on average. Therefore there exists a choice of $A'$ which produces a $P$-free subgraph of size $\Omega(m^{1/2})$.
\end{proof}

The next lemma gives a trivial upper bound on $sm(m, P)$ for $0-1$ matrices $P$ in terms of $ex(n, P)$.

\begin{lem}\label{smex}
$sm(m, P) \leq ex(m^{1/2}, P)$
\end{lem}

\begin{proof}
Let $A$ be an $m^{1/2} \times m^{1/2}$ matrix with every entry equal to $1$. Then the maximum number of ones in any $P$-free $0-1$ matrix contained in $A$ is $ex(m^{1/2}, P)$.
\end{proof}

This lemma can be used to derive a number of corollaries about $sm(m, P)$.

\begin{cor}\label{linex}
If $P$ is a $0-1$ matrix with $ex(n, P) = O(n)$, then $sm(m, P) = O(m^{1/2})$.
\end{cor}

\begin{cor}\label{light}
If $P$ is a $0-1$ matrix with a single one in every column, then $sm(m, P)$ is quasilinear in $m^{1/2}$.
\end{cor}

\begin{proof}
It is known that for every such $P$, there exists a constant $t$ such that $ex(n, P) \leq n 2^{\alpha(n)^{t}}$ \cite{klaz}. Since $ex(n, P)$ is quasilinear in $n$, then $sm(m, P)$ is quasilinear in $m^{1/2}$ by Lemma \ref{smex}.
\end{proof}

\begin{cor}
If $P = \begin{bmatrix} 1 & 0 \\ 1 & 1 \end{bmatrix}$, then $sm(m, P) = O(m^{1/2})$.
\end{cor}

\begin{proof}
Since $ex(n, P) = O(n)$, then the result follows from Corollary \ref{linex}.
\end{proof}

The next two lemmas show that for every $0-1$ matrix $P$ with $ex(n, P) = O(n)$, either $sm(m, P) = \Theta(1)$ or $sm(m, P) = \Theta(m^{1/2})$.

\begin{lem}
If $P$ is either the $k \times 1$ matrix with all ones, the $1 \times k$ matrix with all ones, or a $k \times k$ matrix with $k$ ones on a diagonal, then $sm(m, P) = k-1$.
\end{lem}

\begin{proof}
The upper bound follows by considering either the $m \times 1$ matrix with all ones, the $1 \times m$ matrix with all ones, or an $m \times m$ matrix with $m$ ones on a diagonal. The lower bound follows since $P$ has $k$ ones.
\end{proof}

The next lemma gives a more general bound than the one in Lemma \ref{lshape}.

\begin{lem}\label{3cases}
Suppose that $P$ is a $0-1$ matrix that contains at least one of the following patterns:
$\begin{bmatrix} 0 & 1 & 0 \\ 1 & 0 & 1 \end{bmatrix}, \begin{bmatrix} 0 & 0 & 1 \\ 1 & 1 & 0 \end{bmatrix}, \begin{bmatrix} 0 & 1 \\ 1 & 1 \end{bmatrix}, \begin{bmatrix} 0 & 0 & 1 \\ 1 & 0 & 0 \\ 0 & 1 & 0 \end{bmatrix}$

Then $sm(m, P) = \Omega(m^{1/2})$.
\end{lem}

\begin{proof}
Let $A$ be a $0-1$ matrix with $m$ ones. Construct a sequence $S$ by the following process: Start $S$ as the empty sequence. For each row in $A$ from top to bottom, scan the current row from left to right, and append the integer $i$ to $S$ if column $i$ has a $1$ in the current row. By the Erdos-Szekeres theorem, $S$ contains either a non-increasing or non-decreasing subsequence of length $m^{1/2}$. This subsequence corresponds to a $0-1$ matrix contained in $A$ that avoids $P$ and has at least $m^{1/2}$ ones.
\end{proof}

\begin{cor}
Suppose that $P$ is a matrix such that $ex(n, P) = O(n)$. Then $sm(m,P) = \Theta(m^{1/2})$ if for every $k$, $P$ is not contained in any of the following matrices: the $k \times 1$ matrix with all ones, the $1 \times k$ matrix with all ones, or a $k \times k$ matrix with $k$ ones on a diagonal. Otherwise, $sm(m,P) = \Theta(1)$.
\end{cor}

\section{Sequence Results}

\begin{lem}
If $v$ is a sequence equal to either $k$ copies of the same letter or $k$ distinct letters, then $ss(m, v) = k-1$.
\end{lem}

\begin{proof}
The upper bound follows by considering the sequence equal to either $m$ copies of the same letter or $m$ distinct letters. The lower bound follows since $v$ has $k$ letters.
\end{proof}

\begin{lem}\label{aaabl}
If $v$ contains both $a a$ and $a b$, then $ss(m, v) \geq m^{1/2}$.
\end{lem}

\begin{proof}
Let $u$ be any sequence of length $m$. If there is no letter that occurs at least $m^{1/2}$ times in $u$, then there are at least $m^{1/2}$ distinct letters in $u$.
\end{proof}

\begin{lem}
$ss(m, a b a b) = \Theta(m^{1/2})$.
\end{lem}

\begin{proof}
The lower bound follows from Lemma \ref{aaabl}. For the upper bound, let $m = k^{2}$ be a perfect square and consider the sequence $u$ of length $m$ such that $u = (a_{1} \ldots a_{k})^{k}$. Let each disjoint copy of $(a_{1} \ldots a_{k})$ in $u$ be called a block.

Let $v$ be a subsequence of $u$ of length $3k$ and suppose for contradiction that $v$ avoids $a b a b$. Every subsequence of length $k+1$ has a subsequence isomorphic to $a a$. Therefore there are $k$ disjoint copies of $a a$ in $v$. The first copy of $a a$ is in two blocks of $u$, and each successive copy requires at least one more block since $v$ avoids $a b a b$. This is a contradiction.
\end{proof}

\begin{cor}
$sp(m,2) = \Theta(m^{1/2})$
\end{cor}

\begin{cor}
$ss(m, a b a) = \Theta(m^{1/2})$ and $sp(m,1) = \Theta(m^{1/2})$
\end{cor}

In order to bound $sp(m, i)$ for $i > 2$, we show a general bound on $ss(m, v)$ for every sequence $v$. For these bounds, we use the same fact as in Lemma \ref{light} about $0-1$ matrices.

\begin{lem}
For every sequence $v$, $ss(m, v)$ is quasilinear in $m^{1/2}$.
\end{lem}

\begin{proof}
Let $m = k^{2}$ be a perfect square and consider the sequence $u$ of length $m$ such that $u = (a_{1} \ldots a_{k})^{k}$. Again, let each disjoint copy of $(a_{1} \ldots a_{k})$ in $u$ be called a block.

Let $s$ be any subsequence of $u$ that avoids $v$. Create a $0-1$ matrix $A$ from $s$ with $k$ columns and $k$ rows that has a one in row $i$ and column $j$ if $s$ has letter $a_{i}$ in block $j$. 

Let $t$ be the length of $v$, let $r$ be the number of distinct letters in $v$, and let the letters of $v$ be $l_{1},\ldots,l_{r}$. Let $P$ be the $r \times t$ $0-1$ matrix such that $P$ has a one in row $i$ and column $j$ if the $j^{th}$ letter of $v$ is $l_{i}$. Then $A$ avoids $P$, or else $s$ would contain $v$. Since $ex(n, P)$ is quasilinear in $n$ as in Lemma \ref{light}, $ss(m, v)$ is quasilinear in $m^{1/2}$.
\end{proof}

\begin{cor}
$sp(m, k)$ is quasilinear in $m^{1/2}$ for every $k > 0$.
\end{cor}

\section{Open Problems}

Many questions are left unanswered about the functions $sm$, $ss$, and $sp$. For example, what are the $0-1$ matrices $P$ such that $sm(m, P) = \Theta(m^{1/2})$? In this paper, we identified the matrices $P$ for which $ex(n, P) = O(n)$ and $sm(m, P) \neq \Theta(m^{1/2})$. Are there any matrices $P$ for which $ex(n, P) = \omega(n)$ and $sm(m, P) = \Theta(m^{1/2})$?

We also showed that $ss(m, a b a) = \Theta(m^{1/2})$ and $ss(m, a b a b) = \Theta(m^{1/2})$, but the bounds are not tight for alternations of length at least $5$. If $u$ is an alternation of length at least $5$, then $ss(m, u)$ has an obvious lower bound of $\Omega(m^{1/2})$ and an upper bound that is quasilinear in $m^{1/2}$. Is there a lower bound on $ss(m, u)$ that is nonlinear in $m^{1/2}$?

The same questions about $ss$ and alternations can be asked about $sp$ and polynomials of degree at least $3$. We showed that $sp(m, 1) = \Theta(m^{1/2})$ and $sp(m, 2) = \Theta(m^{1/2})$, but the bounds on $sp(m, k)$ are not tight for $k > 2$. As with $ss$, the upper bounds on $sp(m, k)$ are quasilinear in $m^{1/2}$ for $k > 2$. Is there a lower bound on $sp(m, k)$ that is nonlinear in $m^{1/2}$ for $k > 2$?

Another problem is to find a faster algorithm for computing $LSP$: given a lower envelope in the form of an arbitrary sequence $u$, what is the maximum length of any subsequence of $u$ that can be realized as the lower envelope of a set of polynomials of degree at most $d$? There are obvious deterministic brute-force algorithms for $LSP$, but what is the minimum worst-case time complexity for a deterministic algorithm to solve $LSP$? The same question can be asked for $LSM$ and $LSS$. In addition, are there efficient algorithms for approximating $LSP$, $LSS$, and $LSM$ up to an additive or multiplicative factor?

\end{document}